\documentclass[11pt]{article}
\usepackage[margin=1in]{geometry} 
\usepackage{amssymb,amsfonts,amsmath,bbm,mathrsfs,stmaryrd,mathtools}
\usepackage{xcolor}
\usepackage{url}

\usepackage{extarrows}

\usepackage[shortlabels]{enumitem}
\usepackage{tensor}

\usepackage{xr}
\usepackage[T1]{fontenc}

\usepackage[colorlinks,
linkcolor=black!75!red,
citecolor=blue,
pdftitle={},
pdfproducer={pdfLaTeX},
pdfpagemode=None,
bookmarksopen=true,
bookmarksnumbered=true,
backref=page]{hyperref}

\usepackage{tikz}
\usetikzlibrary{arrows,calc,decorations.pathreplacing,decorations.markings,intersections,shapes.geometric,through,fit,shapes.symbols,positioning,decorations.pathmorphing}

\usepackage{braket}

\usepackage[amsmath,thmmarks,hyperref]{ntheorem}
\usepackage{cleveref}

\creflabelformat{enumi}{#2#1#3}

\crefname{section}{Section}{Sections}
\crefformat{section}{#2Section~#1#3} 
\Crefformat{section}{#2Section~#1#3} 

\crefname{subsection}{\S}{\S\S}
\AtBeginDocument{%
  \crefformat{subsection}{#2\S#1#3}%
  \Crefformat{subsection}{#2\S#1#3}% 
}

\crefname{subsubsection}{\S}{\S\S}
\AtBeginDocument{%
  \crefformat{subsubsection}{#2\S#1#3}%
  \Crefformat{subsubsection}{#2\S#1#3}% 
}

\theoremstyle{plain}

\newtheorem{lemma}{Lemma}[section]
\newtheorem{proposition}[lemma]{Proposition}
\newtheorem{corollary}[lemma]{Corollary}
\newtheorem{theorem}[lemma]{Theorem}

\theoremstyle{plain}
\theoremnumbering{Alph}
\newtheorem{theoremN}{Theorem}

\theoremstyle{plain}
\theorembodyfont{\upshape}
\theoremsymbol{\ensuremath{\blacklozenge}}

\newtheorem{definition}[lemma]{Definition}
\newtheorem{example}[lemma]{Example}
\newtheorem{examples}[lemma]{Examples}
\newtheorem{remark}[lemma]{Remark}
\newtheorem{remarks}[lemma]{Remarks}

\newtheorem{notation}[lemma]{Notation}
\newtheorem{construction}[lemma]{Construction}

\crefname{definition}{definition}{definitions}
\crefformat{definition}{#2definition~#1#3} 
\Crefformat{definition}{#2Definition~#1#3} 

\crefname{ex}{example}{examples}
\crefformat{example}{#2example~#1#3} 
\Crefformat{example}{#2Example~#1#3} 

\crefname{exs}{example}{examples}
\crefformat{examples}{#2example~#1#3} 
\Crefformat{examples}{#2Example~#1#3} 

\crefname{remark}{remark}{remarks}
\crefformat{remark}{#2remark~#1#3} 
\Crefformat{remark}{#2Remark~#1#3} 

\crefname{remarks}{remark}{remarks}
\crefformat{remarks}{#2remark~#1#3} 
\Crefformat{remarks}{#2Remark~#1#3} 

\crefname{convention}{convention}{conventions}
\crefformat{convention}{#2convention~#1#3} 
\Crefformat{convention}{#2Convention~#1#3} 

\crefname{notation}{notation}{notations}
\crefformat{notation}{#2notation~#1#3} 
\Crefformat{notation}{#2Notation~#1#3} 

\crefname{table}{table}{tables}
\crefformat{table}{#2table~#1#3} 
\Crefformat{table}{#2Table~#1#3}

\crefname{lemma}{lemma}{lemmas}
\crefformat{lemma}{#2lemma~#1#3} 
\Crefformat{lemma}{#2Lemma~#1#3} 

\crefname{proposition}{proposition}{propositions}
\crefformat{proposition}{#2proposition~#1#3} 
\Crefformat{proposition}{#2Proposition~#1#3} 

\crefname{propositionN}{proposition}{propositions}
\crefformat{propositionN}{#2proposition~#1#3} 
\Crefformat{propositionN}{#2Proposition~#1#3} 

\crefname{corollary}{corollary}{corollaries}
\crefformat{corollary}{#2corollary~#1#3} 
\Crefformat{corollary}{#2Corollary~#1#3} 

\crefname{corollaryN}{corollary}{corollaries}
\crefformat{corollaryN}{#2corollary~#1#3} 
\Crefformat{corollaryN}{#2Corollary~#1#3} 

\crefname{theorem}{theorem}{theorems}
\crefformat{theorem}{#2theorem~#1#3} 
\Crefformat{theorem}{#2Theorem~#1#3} 

\crefname{theoremN}{theorem}{theorems}
\crefformat{theoremN}{#2theorem~#1#3} 
\Crefformat{theoremN}{#2Theorem~#1#3} 

\crefname{enumi}{}{}
\crefformat{enumi}{#2#1#3}
\Crefformat{enumi}{#2#1#3}

\crefname{assumption}{assumption}{Assumptions}
\crefformat{assumption}{#2assumption~#1#3} 
\Crefformat{assumption}{#2Assumption~#1#3} 

\crefname{construction}{construction}{Constructions}
\crefformat{construction}{#2construction~#1#3} 
\Crefformat{construction}{#2Construction~#1#3} 

\crefname{equation}{}{}
\crefformat{equation}{(#2#1#3)} 
\Crefformat{equation}{(#2#1#3)}

\numberwithin{equation}{section}
\renewcommand{\theequation}{\thesection-\arabic{equation}}

\theoremstyle{nonumberplain}
\theoremsymbol{\ensuremath{\blacksquare}}

\newtheorem{proof}{Proof}
\newcommand\pf[1]{\newtheorem{#1}{Proof of \Cref{#1}}}

\newcommand\bZ{{\mathbb Z}}

\newcommand\cC{{\mathcal C}}

\newcommand\cI{{\mathcal I}}

\newcommand\cL{{\mathcal L}}
\newcommand\cM{{\mathcal M}}

\newcommand\ol{\overline}

\DeclareMathOperator{\id}{id}

\DeclareMathOperator{\soc}{\mathrm{soc}}

\DeclareMathOperator{\spn}{\mathrm{span}}

\DeclareMathOperator{\Coend}{\mathrm{Coend}}

\newcommand\numberthis{\addtocounter{equation}{1}\tag{\theequation}}

\newcommand{\cat}[1]{\textsc{#1}}

\newcommand{\qedhere}{\mbox{}\hfill\ensuremath{\blacksquare}}

\newcommand{\xrightarrowdbl}[2][]{%
  \xrightarrow[#1]{#2}\mathrel{\mkern-14mu}\rightarrow
}

\title{Prescribed duality dynamics in comodule categories}
\author{Alexandru Chirvasitu}

\begin{document}

\date{}

\newcommand{\Addresses}{{% additional braces for segregating \footnotesize
  \bigskip
  \footnotesize

  \textsc{Department of Mathematics, University at Buffalo}
  \par\nopagebreak
  \textsc{Buffalo, NY 14260-2900, USA}  
  \par\nopagebreak
  \textit{E-mail address}: \texttt{achirvas@buffalo.edu}

  % % \medskip
  % % 
  % % \textsc{Department of Mathematics, INSTITUTION}
  % % \par\nopagebreak
  % % \textsc{ADDRESS}
  % % \par\nopagebreak
  % % \textit{E-mail address}: \texttt{??}
  % % 

}}

\maketitle

\begin{abstract}
  We prove that there exist Hopf algebras with surjective, non-bijective antipode which admit no non-trivial morphisms from Hopf algebras with bijective antipode; in particular, they are not quotients of such. This answers a question left open in prior work, and contrasts with the dual setup whereby a Hopf algebra has injective antipode precisely when it embeds into one with bijective antipode. The examples rely on the broader phenomenon of realizing pre-specified subspace lattices as comodule lattices: for a finite-dimensional vector space $V$ and a sequence $(\mathcal{L}_r)_r$ of successively finer lattices of subspaces thereof, assuming the minimal subquotients of the supremum $\bigvee_r \mathcal{L}_r$ are all at least 2-dimensional, there is a Hopf algebra equipping $V$ with a comodule structure in such a fashion that the lattice of comodules of the $r^{th}$ dual comodule $V^{r*}$ is precisely the given $\mathcal{L}_r$.
\end{abstract}

\noindent {\em Key words:
  Diamond Lemma;
  Tannaka reconstruction;
  adjoint functor;
  antipode;
  comodule;
  free Hopf algebra;
  subquotient;
  triangular
  
}

\vspace{.5cm}

\noindent{MSC 2020:
  16S10;
  16T05;
  16T15;
  16T30;
  18A40;
  18M05

}

%\tableofcontents

%%%%%%%%%%%%%%%%%%%%%%%%%%%%%%%%%%%%%%%%%%%%%%%%%%%%%%%%%%%%%%%%%%%%%%%%%%%%%
%%%%%%%%%%%%%%%%%%%%%%%%%%%%%%%%%%%%%%%%%%%%%%%%%%%%%%%%%%%%%%%%%%%%%%%%%%%%%
\section*{Introduction}

Consider an object in an abelian {\it left rigid monoidal category} $\cC$, i.e. (\cite[Definition 3.3.1]{par_qg-ncg}, \cite[\S 2.10]{egno}) on in which each object $V$ has a {\it left dual} $V^*$ \cite[Definition 2.10.1]{egno}: equipped with morphisms
\begin{equation*}
  V^*\otimes V\xrightarrow{\quad e\text{ (evaluation)}\quad}\text{monoidal unit }{\bf 1}
  ,\quad
  {\bf 1}
  \xrightarrow{\quad c\text{ (coevaluation)}\quad}
  V\otimes V^*
\end{equation*}
with
\begin{equation*}
  \begin{tikzpicture}[>=stealth,auto,baseline=(current  bounding  box.center)]
    \path[anchor=base] 
    (0,0) node (l) {$V$}
    +(2,.5) node (u) {$V\otimes V^*\otimes V$}
    +(4,0) node (r) {$V$}

    +(5,0) node (lr) {$V^*$}
    +(7,.5) node (ur) {$V^*\otimes V\otimes V^*$}
    +(9,0) node (rr) {$V^*$}
    ;
    \draw[->] (l) to[bend left=6] node[pos=.5,auto] {$\scriptstyle c\otimes \id$} (u);
    \draw[->] (u) to[bend left=6] node[pos=.5,auto] {$\scriptstyle \id\otimes e$} (r);
    \draw[->] (l) to[bend right=6] node[pos=.5,auto,swap] {$\scriptstyle \id$} (r);
    \draw[->] (lr) to[bend left=6] node[pos=.5,auto] {$\scriptstyle \id\otimes c$} (ur);
    \draw[->] (ur) to[bend left=6] node[pos=.5,auto] {$\scriptstyle e\otimes \id$} (rr);
    \draw[->] (lr) to[bend right=6] node[pos=.5,auto,swap] {$\scriptstyle \id$} (rr);
  \end{tikzpicture}
\end{equation*}
The (contravariant) functorial nature \cite[post Example 2.10.14]{egno} of $V\mapsto V^*$ implies that subobjects (quotients) of $V$ produce quotients (respectively subobjects) of $V^*$. In short, the subobject lattice of $V$ embeds into the dual to that of $V^*$. That embedding is in general proper (we will recall familiar examples presently), so that iterated duality generates increasingly complicated subobject lattices for the objects $V^{r*}:=(V^{(r-1)*})^*$; this is the ``duality dynamics'' of the paper's title, and one motivating issue is to determine to what extent that ``branching'' behavior (in the sense that simple objects might acquire non-trivial subobjects after dualization) can be controlled. 

The only rigid categories featuring below are those of the form $\cM^H_f$, finite-dimensional (right) comodules over Hopf $\Bbbk$-algebras $H$ for a field $\Bbbk$. Left rigidity is implemented by means of the antipode $S$ of $H$ \cite[Remark 5.3.8]{egno}, equipping the usual dual vector space $V^*$ of a right $H$-comodule
\begin{equation*}
  V
  \ni
  v  
  \xmapsto{\quad}
  v_0\otimes v_1
  \in
  V\otimes H
  \quad
  (\text{{\it Sweedler notation} \cite[\S 2.0, pp.32-33]{swe}})
\end{equation*}
with the comodule structure defined implicitly by
\begin{equation*}
  \braket{v^*,v_0}Sv_1
  =
  \braket{v^*_0,v}v^*_1
  \in H
  ,\quad
  \forall v\in V,\ v^*\in V^*
  ,\quad
  V^*\otimes V\xrightarrow{\text{usual evaluation }\braket{-,-}}\Bbbk.
\end{equation*}
As noted in the same \cite[Remark 5.3.8]{egno}, this will not, typically, make $V^*$ into a {\it right} dual \cite[Definition 2.10.2]{egno} (or {\it predual}): the usual evaluation is a comodule morphism when defined on $V^*\otimes V$, but not on $V\otimes V^*$; right duals, rather, can be defined analogous using the inverse antipode $\overline{S}$ when it exists (it doesn't always \cite[Theorem 11]{zbMATH03344702}). 

The ``control'' alluded to two paragraphs up means the degree to which duality can be expected to branch $V$ to just the desired extent and no more. To illustrate the limitations of what can be expected (again, focusing exclusively on categories of comodules), consider an $H$-comodule $V$ whose dual $V^*$ has a 1-dimensional subcomodule $L\le V^*$. $V$ must then surject onto $L^*$ (as elaborated in \Cref{re:needge2} below; the issue is that left and right duals coincide for 1-dimensional Hopf-algebra comodules), so it could not, for instance, have been irreducible (if at least $2$-dimensional). 

In part, the goal is to confirm that this is essentially the only difficulty. Extend the notation $V^{r*}$ employed above to negative $r$, in which case it denotes successive right duals of $V$. As noted, the distinction does matter once we equip spaces with comodule structures. \Cref{cor:cr2h} to \Cref{th:basis4triangchn} yields, via \Cref{con:comod2coalg}, the following conclusion. 

\begin{theoremN}\label{thn:prscrflg}
  Let $V$ be a finite-dimensional vector space and $(\cL_r)_{r\ge d}$ a sequence of increasingly finer lattices of subspaces of $V$ for some $d\in \bZ_{\le 0}\sqcup\{-\infty\}$.

  If the minimal subquotients of the supremum lattice $\bigvee_r \cL_r$ are all of dimension $\ge 2$ there is a Hopf algebra $H$ equipping $V^{r*}$ with comodule structures
  \begin{equation*}
    V^{r*}
    \xrightarrow{\quad\rho_r\quad}
    V^{r*}\otimes H
    ,\quad
    r\in \bZ_{\ge d}
  \end{equation*}
  so that
  \begin{itemize}[wide]
  \item $\rho_{r+1}$ is (left) dual to $\rho_r$;

  \item and for every $r$ the lattice of $H$-subcomodules of $V^{r*}$ is precisely $\cL_r^{r*}$ (meaning $\cL_r$ if $r$ is even and its dual if not).  \qedhere
  \end{itemize}
\end{theoremN}

It is not unnatural, at this point, to link back to antipode bijectivity, given the already-mentioned relationship between inverse antipodes and right duality (we refer the reader once more to \cite[Remark 5.3.8]{egno}). The following observations will serve as a starting point:

\begin{itemize}[wide]
\item Plainly, Hopf algebras surjected upon (embeddable into) those with bijective antipode have surjective (respectively injective) antipode.

\item One of the converses holds: a Hopf algebra with injective antipode embeds into one with bijective antipode \cite[Proposition 2.7]{schau-ff} (universally, i.e. the embedding is the {\it unit} \cite[Definition 19.3]{ahs} of an adjunction between Hopf algebras and Hopf algebras with bijective antipode). 
\end{itemize}

The inevitable problem of whether the other converse holds is posed as \cite[Question 3]{zbMATH05696924}. The first example \cite[\S 3]{schau-ff} of a Hopf algebra with surjective, non-injective antipode is explicitly constructed as a quotient of a free Hopf algebra with bijective antipode, so is tautologically not a counterexample to an affirmative answer. 

The connection to the preceding material is through the {\it coefficient coalgebras} $C_{W}$ attached to $H$-comodules $W$ (we recall the notion more extensively in \Cref{con:comod2coalg}): the smallest subcoalgebra for which the comodule structure factors as
\begin{equation*}
  \begin{tikzpicture}[>=stealth,auto,baseline=(current  bounding  box.center)]
    \path[anchor=base] 
    (0,0) node (l) {$W$}
    +(2,.5) node (u) {$W\otimes C_{W}$}
    +(4,0) node (r) {$W\otimes H$.}
    ;
    \draw[->] (l) to[bend left=6] node[pos=.5,auto] {$\scriptstyle $} (u);
    \draw[->] (u) to[bend left=6] node[pos=.5,auto] {$\scriptstyle $} (r);
    \draw[->] (l) to[bend right=6] node[pos=.5,auto,swap] {$\scriptstyle $} (r);
  \end{tikzpicture}
\end{equation*}
Given that $C_{W^*}=S(C_W)$, the antipode fails to be injective precisely to the extent to which duality enlarges subspace lattices. This suggests that if the dynamics of a surjective antipode is ill-behaved enough, perhaps the Hopf algebra in question will fail to be surjected upon by one with bijective antipode. \Cref{ex:notquotbij} confirms this (in a stronger form than just stated), as a consequence of \Cref{cor:notsurjfrombij} to \Cref{th:largedimatinfty}.

\begin{theoremN}\label{thn:nosurjfrombij}
  There exist Hopf algebras with surjective antipode which do not admit non-trivial morphisms from Hopf algebras with bijective antipode.  \qedhere
\end{theoremN}

% % %%%%%%%%%%%%%%%%%%%%%%%%%%%%%%%%%%%%%%%%%%%%%%%%%%%%%%%%%%%%%%%%%%%%%%%%%%%%%
% % \subsection*{Acknowledgements}
% % 

% % %%%%%%%%%%%%%%%%%%%%%%%%%%%%%%%%%%%%%%%%%%%%%%%%%%%%%%%%%%%%%%%%%%%%%%%%%%%%%
% % %%%%%%%%%%%%%%%%%%%%%%%%%%%%%%%%%%%%%%%%%%%%%%%%%%%%%%%%%%%%%%%%%%%%%%%%%%%%%
% % \section{Preliminaries}\label{se:prel}
% %

%%%%%%%%%%%%%%%%%%%%%%%%%%%%%%%%%%%%%%%%%%%%%%%%%%%%%%%%%%%%%%%%%%%%%%%%%%%%%
%%%%%%%%%%%%%%%%%%%%%%%%%%%%%%%%%%%%%%%%%%%%%%%%%%%%%%%%%%%%%%%%%%%%%%%%%%%%%
\section{Free Hopf algebras on coalgebra chains and main results}\label{se:main}

(Co)algebras are (co)unital, (co)associative and, along with everything else, $\Bbbk$-linear for some field $\Bbbk$. (Co)op superscripts indicate {\it (co)opposite} (co)multiplications \cite[Definitions 2.1.1 and 2.1.4]{rad}. 

We take for granted the various universal constructions peppered throughout (free objects and more general types of adjoint functors), as justified by the fact that all categories involved are {\it locally presentable} in the sense of \cite[Definition 1.17]{ar}.  Sources for this, for the categories featuring most prominently below (algebras $\cat{Alg}$, coalgebras $\cat{Coalg}$, bialgebras $\cat{BiAlg}$, Hopf algebras $\cat{HAlg}$ or Hopf algebras with bijective antipode $\cat{HAlg}_{\overline{S}}$) and more, include (for instance) \cite[Lemmas 1 and 2, Theorem 6 and Proposition 22 4.]{zbMATH06696043}. 

We will need the following variant of the {\it free Hopf algebra} $H(C)$ on a coalgebra $C$ \cite[Definition 2]{zbMATH03344702}, i.e. the left adjoint to the forgetful functor $\cat{HAlg}\to \cat{Coalg}$. 

\begin{definition}\label{def:freeonchain}
  Let $d\in \bZ_{\le 0}\sqcup \{-\infty\}$ be a non-positive extended integer.

  \begin{enumerate}[(1),wide]
  \item\label{item:def:freeonchain:skch} A {\it skew chain} $(C^{r},\theta^r)_{i\ge d}$ of coalgebras is a diagram
    \begin{equation*}
      \cdots
      \xrightarrow{\quad\theta^{-1}\quad}
      C^{0}
      \xrightarrow{\quad\theta^0\quad}
      (C^{1})^{cop}
      \xrightarrow{\quad\theta^1\quad}
      C^{2}
      \xrightarrow{\quad\theta^2\quad}
      \cdots
      \quad\text{in \cat{Coalg}},
    \end{equation*}
    with the co-opposites alternating. We write $\cat{Coalg}_{\ge d}$ for the category of skew coalgebra chains. 
        
  \item\label{item:def:freeonchain:freeonskch} The {\it free Hopf algebra} $H({\bf C})$ on a skew chain ${\bf C}=(C^{r},\theta^r)$ is (the image of ${\bf C}$ through) the left adjoint to the forgetful functor
    \begin{equation*}
      \cat{HAlg}
      \ni H
      \xmapsto{\ \cat{fgt}\ }
      \left(
        \cdots
        \xrightarrow{\ \text{antipode }S\ }
        H
        \xrightarrow{\quad S\quad}
        H^{op,cop}
        \xrightarrow{\quad S\quad}
        H
        \xrightarrow{\quad S\quad}
        \cdots
      \right)
      \in
      \cat{Coalg}_{\ge d}
    \end{equation*}
    $H:=H(C^{r},\theta^r)$ naturally comes equipped with coalgebra morphisms $C^{r}\xrightarrow{\iota_r} H$ intertwining the antipode on the codomain and the $\theta^r$ on the domain(s). 
  \end{enumerate}
\end{definition}

\begin{remark}\label{re:skewchaingoodcat}
  Mapping from $C^{r}$ to $(C^{r+1})^{cop}$ rather than $C^{r+1}$ is only a matter of choice in labeling; the category is of course equivalent to that of functors $\bZ_{\ge d}\to \cat{Coalg}$, with the domain, a poset, regarded as a category as usual \cite[Example 1.2.6.b]{brcx_hndbk-1}: exactly one arrow $\to$ for each relation $\le$. As a functor category, it is locally presentable \cite[Corollary 1.54]{ar} along with its codomain $\cat{Coalg}$. In particular, the existence of the left adjoint
  \begin{equation*}
    \cat{Coalg}_{\ge d}
    \xrightarrow{\quad}
    \cat{HAlg}
  \end{equation*}
  taken for granted in \Cref{def:freeonchain}\Cref{item:def:freeonchain:freeonskch} is unproblematic, by any number of {\it adjoint functor theorems} (\cite[\S 18]{ahs}, \cite[\S 0.7]{ar}, etc.). It is also a simple enough matter to give a (relatively) concrete description of $H({\bf C})$, as we do in the proof of \Cref{pr:scalarext}. 
\end{remark}

Instances of the construction of particular interest here:

\begin{examples}\label{exs:recoverprev}
  \begin{enumerate}[(1),wide]
  \item\label{item:exs:recoverprev:free} The aforementioned \cite[Definition 2]{zbMATH03344702} motivating free Hopf algebra $H(C)$ on a coalgebra $C$ can be retrieved as $H(C^{r},\theta^r\ |\ r\ge 0)$ with
    \begin{equation}\label{eq:trivcithetai}
      C^{r}:=
      \begin{cases}
        C
        &\text{if $r$ is even}\\
        C^{cop}
        &\text{otherwise}
      \end{cases}
      \quad\text{and}\quad
      \theta^r=\id,\quad \forall r.
    \end{equation}

  \item\label{item:exs:recoverprev:freebij} In precisely the same fashion one recovers the free Hopf algebra with bijective antipode on $C$ denoted by $\widehat{H}(C)$ in \cite[Lemma 3.1]{schau-ff}. The only difference is that this time the chain is bi-infinite, i.e. $\widehat{H}(C)\cong H(C^{r},\theta^r\ |\ r\in \bZ)$ with $C^{r}$ and $\theta^r$ still as in \Cref{eq:trivcithetai}.

    It will be convenient to abbreviate the description of $C^{r}$ in \Cref{eq:trivcithetai} as $C^{r}=C^{r\cdot cop}$: applying the $cop$ operator an even number of times amounts to doing nothing, and the multiplicative notation $r\cdot cop$ seems less burdensome than the exponential version $cop^r$. 

  \item\label{item:exs:recoverprev:freebijcoll} Even more pertinent to the sorts of phenomena relevant in the present context, consider the quotient
    \begin{equation*}
      H:=\widehat{H}(M_4^*)/I:=\left(\text{Hopf ideal generated by }x^0_{ij},\ i\ge 3,\ j\le 2\right)
    \end{equation*}
    where $M_4^*:=M_4(\Bbbk)^*$ is the coalgebra dual to the matrix algebra, $x^0_{ij}\in M_4^*$ are the basis elements dual to the standard matrix units in $M_4$, and the `$0$' superscript indicates they belong to the initial generating copy of $M_4^*\le \widehat{H}(M_4^*)$ rather than to any antipode iterates thereof. Per the conventions of \cite[\S 3]{schau-ff} (also \cite[\S 2]{nic}),
    \begin{equation*}
      S^r M_4^* = \spn\left\{x^r_{ij}\ |\ 1\le i,j\le 4\right\},\ \forall r\in \bZ
      \quad\text{and}\quad
      S x^{r}_{ij}=x^{r+1}_{ji}.
    \end{equation*}
    We have $H\cong H(C^{r},\theta^r\ |\ r\in \bZ)$ for
    \begin{equation*}
      C^{r}:=
      \begin{cases}
        (M_4^*)^{r\cdot cop}
        &\text{if }r<0\\
        (M_4^*/\spn\left\{x_{ij}\ |\ i\ge 3,\ j\le 2\right\})^{r\cdot cop}
        &\text{otherwise},
      \end{cases}      
    \end{equation*}
    with the obvious morphisms (identities or surjections). In words: the chain, traveling rightward along identities from $-\infty$, collapses $M_4^*$ to its quotient by the four matrix counits at step $0$ and stabilizes afterwards. 
  \end{enumerate}
\end{examples}

As a first general observation, just as the free Hopf algebra construction $H(-)$ (per \cite[Corollary 8]{zbMATH03344702}), free Hopf algebras on skew chains are compatible with scalar extensions.

\begin{proposition}\label{pr:scalarext}
  For any skew coalgebra chain ${\bf C}=(C^r,\theta^r)$ over $\Bbbk$ and field extension $\Bbbk\le \Bbbk'$ the canonical morphism $H({\bf C})\otimes \Bbbk'\to H({\bf C}\otimes \Bbbk')$ for
  \begin{equation*}
    {\bf C}\otimes \Bbbk'
    :=
    \left(C^r\otimes \Bbbk',\theta^r\otimes\id_{\Bbbk'}\right)
  \end{equation*}
  is an isomorphism. 
\end{proposition}
\begin{proof}
  An explicit description of $H({\bf C})$:
  \begin{itemize}[wide]
  \item form the usual free coalgebra $H\left(\bigoplus_r C^r\right)$, equipped with its injective \cite[Corollary 9]{zbMATH03344702} structure map
    \begin{equation*}
      \bigoplus C^r
      \lhook\joinrel\xrightarrow{\quad\iota\quad}
      H\left(\bigoplus C^r\right),
    \end{equation*}
    henceforth suppressed, identifying $C^r$ with subcoalgebras of its codomain;

  \item and quotient out the ideal generated by $\theta^r(x^r)-Sx^r$ for $x^r\in C^r$ and all of their iterates under $S$. 
  \end{itemize}
  The (algebra) ideal generated by those differences is already a coideal, so the result is indeed a Hopf algebra; the requisite universality property holds tautologically.

  The two bullet points are both invariant under field extensions (the first by the aforementioned \cite[Corollary 9]{zbMATH03344702}), hence the conclusion. 
\end{proof}

\begin{remark}\label{re:canbealgcl}
  In light of \Cref{pr:scalarext} one might occasionally, in analyzing Hopf algebras of the form $H({\bf C})$, assume the ground field algebraically closed. This is sometimes useful for instance, in ensuring that {\it simple} coalgebras (i.e. \cite[2.4.1(1)]{mont} those with no non-obvious subcoalgebras) are matrix coalgebras $M_n^*$. This is dual to the remark \cite[\S 3.5, Corollary b.]{pierce_assoc} that $M_n$ are the only finite-dimensional simple algebras over algebraically closed fields. 
\end{remark}

The collapse noted in \Cref{exs:recoverprev}\Cref{item:exs:recoverprev:freebijcoll} will replicate in a more general setting, requiring some tooling. We will be concerned, specifically, with skew chains $(C^{r},\theta^r)$ where each $C^{r}$ is a dual block-triangular matrix algebra, upper (lower) if $r$ is even (respectively odd), and each $\theta^r$ a surjection refining the block-triangular shape. We elaborate.

\begin{definition}\label{def:triangskch}
  Fix $d\in \bZ_{\le 0}\sqcup\{-\infty\}$, as in \Cref{def:freeonchain}. A {\it (block-upper-)triangular} $\bZ_{\ge d}$-indexed skew coalgebra chain is one of the following form.
  \begin{itemize}[wide]
  \item Each coalgebra $C^{r}$, $r\ge d$ has a basis consisting of {\it matrix counits} $x^r_{ij}$ in the sense that
    \begin{equation*}
      x^r_{ij}
      \xmapsto{\quad\Delta\quad}
      \sum_u x^r_{iu}\otimes x^r_{uj}
      ,\quad
      \varepsilon(x^r_{ij})=\delta_{ij}:=
      \begin{cases}
        1&\text{if }i=j\\
        0&\text{otherwise}.
      \end{cases}
    \end{equation*}
    The $\theta^r$ operate by
    \begin{equation}\label{eq:ij2ji}
      x^r_{ij}
      \xmapsto{\quad\theta^r\quad}
      x^{r+1}_{ji}
      \text{ or }
      0.
    \end{equation}
    We sometimes simply write $x^r_{ij}\mapsto x^{r+1}_{ji}$, understanding that this might vanish if the respective matrix counit is not present in $C^{r+1}$. 
    
  \item The subscripts of $x^{r}_{ij}$ range over a possibly infinite poset $(\cI,\le)$ (common to all $r$), so that \Cref{eq:ij2ji} does indeed make sense. The connected components of the order $\le$ are finite, and for each $r$ the subscripts of $x^{r}_{ij}$ range over a block-triangular pattern with respect to $\le$: upper for even $r$ and lower for odd $r$.
  \end{itemize}
\end{definition}

\begin{remarks}\label{res:coprod}
  \begin{enumerate}[(1),wide]
  \item\label{item:res:coprod:coprod} The conditions of \Cref{def:triangskch} imply in particular that each $C^{r}$ is a coproduct (direct sum)
    \begin{equation*}
      C^{r}\cong \bigoplus_{\text{$\le_r$-connected component }\cI_{rs}} C^{rs}
      ,\quad
      C^{rs}=\spn\left\{x^r_{ij}\ |\ i,j\in \cI_{rs}\right\}
    \end{equation*}
    of finite-dimensional triangular dual matrix coalgebras.

  \item\label{item:res:coprod:simr} Each triangular pattern imposes an equivalence relation $\sim_r$ on $\cI$:
    \begin{equation*}
      i\sim_r j
      \iff
      \text{$x^r_{ij}$ and $x^r_{ji}$ are both present}. 
    \end{equation*}
    The relation $\sim_r$ become progressively finer with increasing $r$.    
  \end{enumerate}
\end{remarks}

\begin{example}\label{ex:just1234}
  In \Cref{exs:recoverprev}\Cref{item:exs:recoverprev:freebijcoll} $(\cI,\le)$ is simply $\overline{1,4}$ with the usual ordering. $\sim_r$ is the coarsest relation (one equivalence class) for $r<0$ has equivalence classes $\{1,2\}$ and $\{3,4\}$ otherwise. 
\end{example}

\begin{notation}\label{not:matunit}
  In the context of \Cref{def:triangskch}:

  \begin{enumerate}[(1),wide]    
  \item\label{item:not:matunit:tuple} Somewhat abusively, we will also denote by $x^r_{ij}$ the images of those elements through the structure maps $C^{r}\xrightarrow{\iota_r} H(C^{r},\theta^r)$; the latter's antipode thus operates by $Sx^r_{ij}=x^{r+1}_{ji}$.

    More generally, following \cite[\S 1, p.103]{zbMATH05807452}, for length-$t$ tuples
    \begin{equation*}
      {\bf r}=(r_s)_{1\le s\le t}
      \quad
      {\bf i}=(i_s)_{1\le s\le t}
      \quad
      {\bf j}=(j_s)_{1\le s\le t}
    \end{equation*}
    set
    \begin{equation*}
      x^{\bf r}_{{\bf ij}}
      :=
      \text{product }
      x^{r_1}_{i_1 j_1}\cdots x^{r_t}_{i_t j_t}
      \text{ in }H(C^{r},\theta^r). 
    \end{equation*}

  \item\label{item:not:matunit:minmax} In discussing products of matrix counits $x^r_{ij}$, the two symbols $\uparrow$ and $\downarrow$ indicate the largest (respectively lowest) row/column index for which the factors respective factors are all non-zero in their respective $C^{r}$. Some examples follow. 
    \begin{itemize}[wide]
    \item $x^r_{i\uparrow}$ is the rightmost non-zero matrix counit on row $i$ in $C^{r}$.

    \item in $x^r_{i\downarrow}x^{s}_{j\downarrow}$ the `$\downarrow$' stands for the lowest index that will make both factors non-zero in $C^{r}$ and $C^{s}$ respectively. Which specific symbol that is depends only on $\max(r,s)$ (as more of the matrix counits vanish in ``later'' $C^{r}$).

    \item \cite[Theorem 5]{nic}, describing a $\Bbbk$-basis for $H(M_n^*)$, can be phrased in the present language: that basis consists of those words $x^{\bf r}_{{\bf i}{\bf j}}$ (for tuples ${\bf r}$ over $\bZ_{\ge 0}$ and ${\bf i}$ over $\overline{1,n}$) which contain no subwords of the form
      \begin{equation*}
        x^r_{i\uparrow}x^{r+1}_{j\uparrow}
        \quad
        x^r_{\uparrow i}x^{r-1}_{\uparrow j}
        \quad
        x^r_{i\uparrow}x^{r+1}_{j(\uparrow-1)}x^{r+2}_{k(\uparrow-1)}
        \quad
        x^r_{\uparrow i}x^{r-1}_{(\uparrow-1)j}x^{r-2}_{(\uparrow-1) k}
      \end{equation*}
    \item Precisely the same goes for the basis described over the course of (the proof of) \cite[Theorem 3.2]{schau-ff} for $\widehat{H}(M_n^*)$, except this time the superscripts range over all of $\bZ$.

      Naturally, in both of these examples $\uparrow=n$ throughout, so there is not much point to the substitution, but in the broader setting below there will be.

    \item \cite[\S 2]{zbMATH06037567} takes a slightly different approach we will find handier here, dispensing with the cubic relations (and the shift in $\uparrow-1$). The bases for $H(M_n^*)$ and $\widehat{H}(M_n^*)$ described there are rather those consisting of words containing none of
      \begin{equation}\label{eq:forbiddenwords}
        x^r_{i\uparrow}x^{r+1}_{j\uparrow}
        \ (r\text{ even})
        \qquad
        x^r_{i\downarrow}x^{r+1}_{j\downarrow}
        \ (r\text{ odd})
        \qquad
        x^{r+1}_{\uparrow i}x^{r}_{\uparrow j}
        \ (r\text{ odd})
        \qquad
        x^{r+1}_{\downarrow i}x^{r}_{\downarrow j}
        \ (r\text{ even})
      \end{equation}
      as subwords. It is this description that the sequel builds upon. 
    \end{itemize}
  \end{enumerate}
\end{notation}

The usefulness of the $x^r_{\updownarrow\updownarrow}$ is evident in the statement of the following result, where the bookkeeping for which indices are involved and which are not might become cumbersome. 

\begin{theorem}\label{th:basis4triangchn}
  Let $(C^{r},\theta^r)_{r\ge d}$ be an upper-triangular skew coalgebra chain in the sense of \Cref{def:triangskch}.

  If the classes of the supremum
  \begin{equation*}
    \bigvee_{r}\sim_r
    ,\quad
    \sim_r=\text{equivalence relation of \Cref{res:coprod}\Cref{item:res:coprod:simr}}
  \end{equation*}
  all have size $\ge 2$ the corresponding $\Bbbk$-Hopf algebra $H(C^{r},\theta^r)$ has a basis consisting of those words in $x^r_{ij}$ containing no subwords of the form \Cref{eq:forbiddenwords}. 
\end{theorem}
\begin{proof}
  As in all cited sources, this will be an application of the {\it Diamond Lemma} \cite[Theorem 1.2]{zbMATH03510490} (which we assume as background, along with its ancillary language and machinery).

  As an algebra, $H:=H(C^{r},\theta^r)$ is the quotient of the free algebra $\Bbbk\braket{x^r_{ij}}$ by the relations making $(x^r_{ij})_{i,j}$ and $(Sx^r_{ij})_{i,j}$ mutual inverses in an appropriately-sized matrix algebra over $H$. Said relations translate to the substitutions (or {\it reductions} \cite[p.180]{zbMATH03510490}) 
  \begin{align*}
    x^r_{i\uparrow}x^{r+1}_{j\uparrow}
    &\xmapsto{\quad}
      \delta_{ij}-\sum_{\alpha<\uparrow} x^r_{i\alpha}x^{r+1}_{j\alpha}
      \quad(\text{$r$ even})\numberthis\label{eq:0up}\\
    x^r_{i\downarrow}x^{r+1}_{j\downarrow}
    &\xmapsto{\quad}
      \delta_{ij}-\sum_{\alpha>\downarrow} x^r_{i\alpha}x^{r+1}_{j\alpha}
      \quad(\text{$r$ odd})\numberthis\label{eq:1up}\\
    x^{r+1}_{\uparrow i}x^{r}_{\uparrow j}
    &\xmapsto{\quad}
      \delta_{ij}-\sum_{\beta<\uparrow} x^{r+1}_{\beta i}x^{r}_{\beta j}
      \quad(\text{$r$ odd})\numberthis\label{eq:1down}\\
    x^{r+1}_{\downarrow i}x^{r}_{\downarrow j}
    &\xmapsto{\quad}
      \delta_{ij}-\sum_{\beta>\downarrow} x^{r+1}_{\beta i}x^{r}_{\beta j}
      \quad(\text{$r$ even})\numberthis\label{eq:0down}\\
  \end{align*}
  These are also \cite[equations (2.1) to (2.4)]{zbMATH06037567} in that paper's more constrained setting; it is also observed there that the four possible types of {\it (overlap) ambiguity} \cite[post Lemma 1.1]{zbMATH03510490} resulting from these reductions are all effectively equivalent.
  \begin{equation}\label{eq:sameamb}
    x^r_{i\uparrow} x^{r+1}_{\downarrow\uparrow} x^r_{\downarrow j}
    \quad
    (\text{$r$ even})
    \quad
    \leftrightsquigarrow
    \quad
    x^{r+1}_{\uparrow i} x^{r}_{\uparrow \downarrow} x^{r+1}_{j \downarrow}
    \quad
    (\text{$r$ odd}),
  \end{equation}
  for instance, in the sense that the computations become symbolically identical, after interchanging the two superscript symbols and the positions of the two subscripts on every $x^{\bullet}_{\bullet\bullet}$. Similar remarks apply to the others, so we fix ideas by focusing on a single one of the four: the right-hand side of \Cref{eq:sameamb}. Note that at this stage we have already made surreptitious use of the hypothesis of having size-$(\ge 2)$ classes: $\uparrow$ and $\downarrow$ are distinct, so there are no ``cubical'' ambiguities $x^r_{i\bullet}x^{r+1}_{j\bullet}x^{r+2}_{k\bullet}$. 

  Resolvability was left to the reader on \cite[p.86]{zbMATH06037567}, so we verify it here for some semblance of completeness.
  \begin{equation*}
    \begin{tikzpicture}[>=stealth,auto,baseline=(current  bounding  box.center)]
      \path[anchor=base] 
      (0,0) node (top) {$
        x^{r+1}_{\uparrow i} x^{r}_{\uparrow \downarrow} x^{r+1}_{j \downarrow}
        $}

      +(-5,-1) node (l1) {$\displaystyle
        \left(
          \delta_{i\downarrow}
          -
          \sum_{\beta<\uparrow} x^{r+1}_{\beta i} x^{r}_{\beta \downarrow}
        \right)
        x^{r+1}_{j \downarrow}
        $}

      +(5,-1) node (r1) {$\displaystyle
        x^{r+1}_{\uparrow i}
        \left(
          \delta_{\uparrow j}
          -
          \sum_{\alpha>\downarrow}x^{r}_{\uparrow \alpha} x^{r+1}_{j \alpha}
        \right)
        $}

      +(-4,-3) node (l2) {$\displaystyle
        \delta_{i\downarrow}x^{r+1}_{j \downarrow}
        -
        \sum_{\beta<\uparrow} \delta_{\beta j} x^{r+1}_{\beta i}
        +
        \sum_{\substack{\beta<\uparrow\\\alpha>\downarrow}}
        x^{r+1}_{\beta i}x^r_{\beta \alpha}x^{r+1}_{j\alpha}
        $}

      +(4,-3) node (r2) {$\displaystyle
        \delta_{\uparrow j}x^{r+1}_{\uparrow i}
        -
        \sum_{\alpha>\downarrow}\delta_{\alpha i}x^{r+1}_{j\alpha}
        +
        \sum_{\substack{\alpha>\downarrow\\\beta<\uparrow}} x^{r+1}_{\beta i} x^r_{\beta\alpha} x^{r+1}_{j\alpha}
        $}
      ;
      \draw[->] (top) to[bend right=6] node[pos=.5,auto,swap] {$\scriptstyle \text{\Cref{eq:1down}}$} (l1);
      \draw[->] (top) to[bend left=6] node[pos=.5,auto] {$\scriptstyle \text{\Cref{eq:1up}}$} (r1);
      \draw[->] (l1) to[bend right=6] node[pos=.2,auto,swap] {$\scriptstyle \text{\Cref{eq:1up}}$} (l2);
      \draw[->] (r1) to[bend left=6] node[pos=.2,auto] {$\scriptstyle \text{\Cref{eq:1down}}$} (r2);
    \end{tikzpicture}
  \end{equation*}
  The cubic sums plainly coincide; we thus have
  \begin{equation*}
    \begin{aligned}
      \text{left}-\text{right}
      &=
        \left(
        \delta_{i\downarrow}x^{r+1}_{j \downarrow}
        +
        \sum_{\alpha>\downarrow}\delta_{\alpha i}x^{r+1}_{j\alpha}
        \right)
        -
        \left(
        \delta_{\uparrow j}x^{r+1}_{\uparrow i}
        +
        \sum_{\beta<\uparrow} \delta_{\beta j} x^{r+1}_{\beta i}
        \right)\\
      &=
        \sum_{\bullet}\delta_{i\bullet}x^{r+1}_{j\bullet}
        -
        \sum_{\bullet}\delta_{\bullet j}x^{r+1}_{\bullet i}\\
      &=
        x^{r+1}_{ji} - x^{r+1}_{ji}=0,
    \end{aligned}    
  \end{equation*}
  effecting the resolution.

  In applications of the Diamond Lemma one typically also produces a {\it monoid partial order} \cite[p.181]{zbMATH03510490} on the free monoid $\braket{X}$ on the set $X$ generators (here $X=\left\{x^r_{ij}\right\}$). One is given in \cite[pp.86-87]{zbMATH06037567} which adapts readily here; however, as pointed out in \cite[\S 5.4]{zbMATH03510490}, over fields (or more generally, in the absence of zero divisors) there is a canonical order that will work provided we know (as we now do, in light of the just-verified resolvability) that all elements are {\it reduction-finite} (i.e. \cite[p.180]{zbMATH03510490} all chains of modifications involving \Cref{eq:0up,eq:1up,eq:1down,eq:0down} eventually stabilize):
  \begin{itemize}[wide]
  \item a word $w\in\braket{X}$ is simply {\it declared} (strictly) smaller than another, $w'$, if $w$ occurs with non-zero coefficient in some reduction of $w'$;
    
  \item the condition on zero divisors will then ensure that the resulting relation $<$ is {\it asymmetric} (i.e. \cite[Tavble 14.1]{gs_disc} $w<w'$ is incompatible with $w'<w$);

  \item and by the assumed reduction finiteness it has the {\it descending chain condition (DCC)} \cite[\S VIII.1]{birkh_latt} (admits no infinite chains $w_0>w_1>\cdots$).
  \end{itemize}
  The validity of the general construction notwithstanding, it certainly functions here: each of the substitutions \Cref{eq:0up,eq:1up,eq:1down,eq:0down} replaces a word with a linear combination of words either strictly shorter or with strictly fewer occurrences of the problematic two-letter left-hand subwords, so that asymmetry and DCC are both self-evident.   
\end{proof}

In particular, because single-letter words of course have no 2-letter subwords, we have the following counterpart to the injectivity \cite[Corollary 9]{zbMATH03344702} of the map $C\to H(C)$:

\begin{corollary}\label{cor:cr2h}
  In the setup of \Cref{th:basis4triangchn} the canonical maps $C^{s}\xrightarrow{\iota_s} H(C^{r},\theta^r)$ of \Cref{def:freeonchain}\Cref{item:def:freeonchain:freeonskch} are injective.  \qedhere
\end{corollary}

Observe also that the only generators entering any of the problematic words \Cref{eq:forbiddenwords} are those featuring in diagonal blocks of both $C^r$ and $C^{r+1}$. For that reason, the ``off-diagonal generators'' simply come along for the ride so to speak, generating an algebra that splits off as a free factor. This is very much analogous to the fact \cite[Theorem 32]{zbMATH03344702} that for any splitting
\begin{equation*}
  C=C_0\oplus V
  ,\quad
  C_0
  :=
  \text{{\it coradical} of }C
  \overset{\text{ \cite[Definition 5.1.5(1)]{mont}}}{:=}
  \sum\left(\text{simple subcoalgebras}\right)
\end{equation*}
we have the algebra-coproduct decomposition $H(C)\cong H(C_0)\coprod TV$ (with $TV$ denoting the tensor algebra of $V$, i.e. the free algebra on $V$). Before stating the present analogue, a piece of terminology.

\begin{definition}\label{def:asympcorad}
  The {\it asymptotic coradical} ${\bf C}_{0\leftarrow}$ of a skew coalgebra chain ${\bf C}=(C^r,\theta^r)$ with surjective connecting maps $\theta^r$ is the {\it filtered} \cite[Definition 1.4]{ar} colimit
  \begin{equation*}
    \varinjlim_{s\ge d}
    {\bf C}_{0\leftarrow}^{\ge s}
    :=
    \left(
      \text{coradical }C^s_{0}
      \xrightarrow{\quad\theta^s\quad}
      \theta^s(C^{s}_0)
      \xrightarrow{\quad\theta^{s+1}\quad}
      (\theta^{s+1}\circ \theta^s)(C^{s}_0)
      \xrightarrow{\quad\theta^{s+2}\quad}
      \cdots
    \right).
  \end{equation*}
  The colimit is taken over decreasing $s$, the connecting maps
  \begin{equation*}
    {\bf C}_{0\leftarrow}^{\ge s}
    \lhook\joinrel\xrightarrow{\quad}
    {\bf C}_{0\leftarrow}^{\ge s-1}
  \end{equation*}
  being the embeddings resulting from the fact \cite[Corollary 5.3.5]{mont} that the image of the coradical $C_0$ through a surjection $C\xrightarrowdbl{}D$ contains the coradical $D_0$. 
\end{definition}

\begin{corollary}\label{cor:offdiag.split}
  For a skew coalgebra chain ${\bf C}=(C^r,\theta^r)$ as in \Cref{th:basis4triangchn} we have an algebra-coproduct decomposition
  \begin{equation*}
    H({\bf C})
    \cong
    H({\bf C}_{0\leftarrow})\coprod TV
    \cong
    H({\bf C}_{0\leftarrow}) \coprod\left(\coprod_r T(V^r)\right)
  \end{equation*}
  for
  \begin{equation*}
    V:=\bigoplus V^r
    ,\quad
    C^r={\bf C}_{0\leftarrow}^r\oplus V^r.
  \end{equation*}  
  \qedhere
\end{corollary}

The passage between the subspace-lattice picture emphasized in the Introduction (\Cref{thn:prscrflg}) and the more coordinate/basis-oriented perspective of \Cref{th:basis4triangchn} is by means of the comodule-coalgebra correspondence familiar from {\it Tannaka reconstruction} (\cite[\S 2]{schau_tann}, \cite[Chapter 3]{par_qg-ncg}, etc.). To expand:

\begin{construction}\label{con:comod2coalg}
  \begin{itemize}[wide]
  \item A subspace lattice $\cL$ on a finite-dimensional vector space $V$ generates a smallest full {\it exact subcategory} \cite[\S 3.4]{freyd_abcats}
    \begin{equation*}
      \cC_{\cL}
      \lhook\joinrel\xrightarrow{\quad\iota_{\cL}\quad}
      \cat{Vect}_f
      :=
      \text{finite-dimensional vector spaces}. 
    \end{equation*}
    The inclusion functor $\iota_{\cL}$ has a {\it coendomorphism coalgebra} \cite[Definition 2.1.8]{schau_tann}, equipped with a surjection
    \begin{equation}\label{eq:coendcoalglatt}
      \Coend(V)\cong M_{\dim V}(\Bbbk)^*
      \xrightarrowdbl{\quad}
      \Coend(\iota_{\cL})
      =:
      C_{\cL};
    \end{equation}
    This is the model for the surjections $C^{r}\to (C^{r+1})^{cop}$ featuring in a triangular skew chain, co-opposites reflecting the dualization in \Cref{thn:prscrflg}.

  \item The finite-dimensional comodule category $\cM^{C_{\cL}}$ is precisely $\cC_{\cL}$ (one version of the aforementioned Tannakian reconstruction \cite[Theorem 2.2.8]{schau_tann}). Subquotients $W$ of $V$ belonging to $\cC_{\cL}$, being $C_{\cL}$-comodules, correspond to subcoalgebras of the latter:
    \begin{equation*}
      W
      \xmapsto{\quad}
      \text{{\it coefficient coalgebra} }C_{W}
      \quad
      \left(\text{associated to the comodule as in \cite[Proposition 2.5.3]{dnr}}\right):
    \end{equation*}
    the smallest subcoalgebra of $C_{\cL}$ for which the structure map $W\to W\otimes C_{\cL}$ takes values in $W\otimes C_W$.

  \item One can now select a basis for $V$, appropriately compatible with $\cL$, so as to ensure a block upper-triangular shape for $C_{\cL}$. The minimal (non-zero) members of $\cL$ sum up to the {\it socle} \cite[Remark 2.4.10(2)]{dnr} of $V\in \cM^{C_{\cL}}$: the largest semisimple sup
    \begin{equation*}
      \soc V
      =
      \bigoplus_i W_i\le V
      ,\quad
      W_i\in \cM^{C_{\cL}}
      \text{ simple}.
    \end{equation*}
    The matrix coalgebras $C_{W_i}\cong \Coend(W_i)\cong M_{\dim W_i}^*$ will be the first diagonal blocks in the (upper) triangular decomposition, with no off-diagonal blocks linking any two:
    \begin{equation*}
      C_{\cL}
      \cong
      \begin{pmatrix}
        C_{W_1}&0&0&\vdots\\
        0&C_{W_2}&0&\vdots\\
        0&0&C_{W_3}&\vdots\\
        \vdots&\vdots&\vdots&\vdots
      \end{pmatrix}
    \end{equation*}

  \item One can then proceed along the {\it socle filtration}
    \begin{equation*}
      \{0\}
      =
      \soc_0 V
      \le
      \soc_1 V:=\soc V
      \le
      \soc_2 V
      \le \cdots
      \text{ of $V$ in }\cM^{C_{\cL}},
    \end{equation*}
    where
    \begin{equation*}
      \text{$s^{th}$ {\it socle layer} }
      \ol{\soc}_{s}
      :=
      \soc_{s}/\soc_{s-1}
      :=
      \soc(W/\soc_{s-1}).
    \end{equation*}
    The off-diagonal blocks of $C_{\cL}$, if any, will link matrix coalgebras $C_{W}$ and $C_{W'}$ for simple comodules $W$ and $W'$ occurring as summands in distinct socle layers.
  \end{itemize}
\end{construction}

\begin{remark}\label{re:needge2}
  The hypothesis on subquotients having dimension $\ge 2$ is necessary in \Cref{cor:cr2h}, as even minimal violations thereof will produce counterexamples.

  Consider for instance an upper-triangular skew chain $(C^r,\theta^r)_{r\ge -1}$ with
  \begin{itemize}
  \item $C^{-1}$ a full $n\times n$ matrix algebra for some $n\ge 2$;

  \item and $C^r$, $r\ge 0$ triangular, each with two diagonal blocks of sizes $1\times 1$ and $(n-1)\times (n-1)$. 
  \end{itemize}
  I then claim that $C^{-1}\xrightarrow{\iota_{-1}}H:=H(C^r,\theta^r)$ cannot be injective (it factors through a triangular quotient of $C^{-1}$, again with one $1\times 1$ diagonal block).

  This can be checked coordinate-wise (much as the argument driving the related \cite[Remark 3.3]{schau-ff} does), but is particularly transparent in monoidal-categorical terms (the connection between the two pictures being as in \Cref{con:comod2coalg}). Denote by $V^r$ the natural $n$-dimensional $C^r$-comodule, resulting from the realization of $C^r$ as a quotient of an $n\times n$ matrix coalgebra. Each $V^r$ then also becomes an $H$-comodule by transporting the structure along $C^r\xrightarrow{\iota_r}H$, and the $1\times 1$ block in $C^0$ means that $V^0$ has a 1-dimensional $C^0$- (hence also $H-$)subcomodule $L$.

  Now, $V^{-1}$ is the predual $\tensor*[^*]{(V^0)}{}$. Over a Hopf algebra, the coefficient coalgebra of a 1-dimensional comodule $L$ is the span of a {\it grouplike} \cite[Definition 1.3.4(a)]{mont} and the latter's inverse \cite[Example 1.5.3]{mont} spans a 1-dimensional comodule that is both left {\it and} right dual to the original $L$. In conclusion, the predual $V^{-1}$ of $V^0$ must surject onto the 1-dimensional $L^*\cong \tensor*[^*]{L}{}$. This will give the image through $C^{-1}\xrightarrow{\iota_{-1}}H$ a lower-triangular structure with a $1\times 1$ rightmost diagonal block.
\end{remark}

We can now return to the statements of the Introduction.

\pf{thn:prscrflg}
\begin{thn:prscrflg}
  I claim that this is essentially what \Cref{cor:cr2h} provides, after producing coalgebras out of subspace lattices as described in \Cref{con:comod2coalg}. Set
  \begin{equation*}
    C^r
    :=
    \begin{cases}
      \text{the coendomorphism coalgebra }C_{\cL_r}\text{ of \Cref{eq:coendcoalglatt}}
      &\text{if $r$ is even}\\
      C_{\cL_r}^{cop}\cong C_{\cL_r^*}
      &\text{otherwise},
    \end{cases}
  \end{equation*}
  with $\cL_r^*$ denoting the lattice on $V^*$ dual to $\cL_r$. The surjections $C^r\xrightarrowdbl{\theta^r}(C^{r+1})^{cop}$ are those analogous to \Cref{eq:coendcoalglatt}, resulting from the refinement $\cL_r\rightsquigarrow \cL^{r+1}$, and \Cref{con:comod2coalg} also outlines how the coalgebras in question acquire a triangular structure.

  This, then, will produce an upper triangular coalgebra chain ${\bf C}=(C^r,\theta^r)$ and a free Hopf algebra $H:=H({\bf C})$ thereon, with the $(\dim\ge 2)$ hypothesis of \Cref{thn:prscrflg} assuring that of \Cref{th:basis4triangchn}. As to the desired conclusion (of \Cref{thn:prscrflg}), it follows from \Cref{cor:cr2h} given that \cite[Lemma 2.2.12]{schau_tann} {\it embeddings} $C\lhook\joinrel\xrightarrow{} D$ of coalgebras induce full, exact inclusions $\cM^C\lhook\joinrel\xrightarrow{} \cM^D$ between the corresponding comodule categories. 
\end{thn:prscrflg}

A brief remark, prefatory to eventually addressing \Cref{thn:nosurjfrombij}:

\begin{lemma}\label{le:rk}
  Let ${\bf C}=(C^r,\theta^r)$ be an upper-triangular skew coalgebra chain meeting the constraints of \Cref{th:basis4triangchn}, and $C\le H({\bf C})$ a simple subcoalgebra.

  There is a unique longest tuple ${\bf r}=(r_1,\cdots,r_t) \in \braket{\bZ_{\ge d}}$ (the free monoid on $\bZ_{\ge d}$) for which the expansion of every non-zero element of $C$ in the basis of \Cref{th:basis4triangchn} contains some word $x^{\bf r}_{{\bf i}{\bf j}}$. 
\end{lemma}
\begin{proof}
  The Hopf algebra $H:=H({\bf C})$ is the sum of its subcoalgebras
  \begin{equation*}
    C^{\bf r}
    :=
    C^{r_1}\cdot C^{r_2}\cdots C^{r_t},
  \end{equation*}
  having identified $C^r$ with a subcoalgebra of $H$ (as \Cref{cor:cr2h} allows). $C$ is thus a subcoalgebra of some $C^{\bf r}$, and the distinguished ${\bf r}$ of the statement will be the {\it minimal} such, length-wise: the relations \Cref{eq:0up,eq:1up,eq:1down,eq:0down} make it clear that $C^{\bf r}$ consists of linear combinations of $x^{\bf r'}_{{\bf i}{\bf j}}$ for subwords ${\bf r'}$ of ${\bf r}$,  hence uniqueness. 
\end{proof}

\begin{definition}\label{def:zd-val.rk}
  The tuple ${\bf r}\in (\bZ_{\ge d})^s$ attached by \Cref{le:rk} to a simple subcoalgebra $C\le H({\bf C})$ is the {\it ($\braket{\bZ_{\ge d}}$-valued) rank} of $C$.

  The term applies also to the (unique up to isomorphism) simple $C$- (hence also $H({\bf C})$-)comodule. 
\end{definition}

\begin{theorem}\label{th:largedimatinfty}
  Let ${\bf C}=(C^r,\theta^r)$ be a skew coalgebra chain as in \Cref{th:basis4triangchn}, ${\bf r}=(r_s)_{1\le s\le t}\in \braket{\bZ_{\ge d}}$, and
  \begin{equation*}
    D^{r_s}\le C^{r_s}
    ,\quad 1\le s\le t
    \quad\text{simple subcoalgebras}.
  \end{equation*}
  Any simple $H({\bf C})$-subcoalgebra of $D^{\bf r}:=\prod_{1\le s\le t}D^{r_s}$ of rank ${\bf r}=(r_1,\cdots,r_t)$ has dimension $\ge\max_s\dim D^{r_s}$.
\end{theorem}
\begin{proof}
  Note that the $D^{r_s}$ are in any case matrix coalgebras, corresponding to square diagonal blocks of $C^{r_s}$ (respectively). There is no harm in assuming the ground field algebraically closed (\Cref{re:canbealgcl}), so that the simple rank-${\bf r}$ $C\le D^{\bf r}$ that concerns us is itself a matrix coalgebra. To fix the notation, we prove only that $\dim C\ge \dim D^{r_1}$. 
 
  The idea behind the proof is that underlying \cite[Proposition 2.6]{zbMATH05807452}. Let $x\in C$, containing the maximal-length term $x^{\bf r}_{{\bf i}{\bf j}}$ in its expansion in the basis of \Cref{th:basis4triangchn}. The $s^{th}$ letter $x^{r_s}_{i_s,j_s}$ of $x^{\bf r}_{{\bf i}{\bf j}}$ is one of the basis elements for the matrix coalgebra $D^{r_s}$, so its subscripts $i_s$ and $j_s$ range independently over a subset
  \begin{equation*}
    \cI_s\subseteq \cI
    ,\quad    
    \left|\cI_s\right|
    =
    n_s
    :=    
    \sqrt{\dim D^{r_s}}.
  \end{equation*}
  We construct $n_1$ length-$t$ tuples ${\bf u}=(u_s)_{1\le s\le t}$ over $\cI$ as follows:
  \begin{itemize}[wide]
  \item pick the first letter $u_1\in \cI_1$ arbitrarily;

  \item then pick the second letter $u_2\in \cI_2$ distinct from $u_1$ (possible by the $(\ge 2)$ assumption transported here from \Cref{th:basis4triangchn});

  \item then the third letter $u_3\in \cI_3$ distinct from $u_2$;
  \item and so on.
  \end{itemize}
  The choice is arbitrary over $\cI_1$ in first instance, and then ensures that no two consecutive letters in ${\bf u}$ coincide. Now note that
  \begin{equation}\label{eq:xrij.comult}
    x^{\bf r}_{{\bf i}{\bf j}}
    \xmapsto{\quad\Delta\quad}
    \sum_{{\bf u}\text{ chosen in this fashion}}
    x^{\bf r}_{{\bf i}{\bf u}}\otimes x^{\bf r}_{{\bf u}{\bf j}}
    +\cdots,
  \end{equation}
  where the $\cdots$ signify tensor products of {\it other} reduced words. The tensorands $x^{\bf r}_{{\bf i}{\bf u}}$ and $x^{\bf r}_{{\bf u}{\bf j}}$ are indeed reduced, by the very shape of the reductions \Cref{eq:0up,eq:1up,eq:1down,eq:0down}: the consecutive letters of ${\bf u}$ are distinct so there is no opportunity for reduction there, whereas the ${\bf i}$ and ${\bf j}$ offer no such opportunity by assumption, $x^{\bf r}_{{\bf i}{\bf j}}$ having been assumed reduced to begin with.

  The outer ${\bf i}$ and ${\bf j}$ on the right-hand side of \Cref{eq:xrij.comult} identify the word on the left-hand side that produced the tensor product by comultiplication, so no other terms in the expansion of $x$ will produce any of the tensor products $x^{\bf r}_{{\bf i}{\bf u}}\otimes x^{\bf r}_{{\bf u}{\bf j}}$. Denoting the `$*$' superscripts the functionals ``dual'' to the word basis of \Cref{th:basis4triangchn}, as in
  \begin{equation*}
    x^{{\bf s}*}_{{\bf u}{\bf v}}
    \left(
      x^{{\bf s'}}_{{\bf u'}{\bf v'}}
    \right)
    =
    \delta_{{\bf r},{\bf r'}}\delta_{{\bf u},{\bf u'}}\delta_{{\bf v},{\bf v'}},
  \end{equation*}
  we have
  \begin{equation*}
    (x^{{\bf r}*}_{{\bf i}{\bf u}}\otimes\id)(\Delta x)
    =
    (\text{coefficient})\cdot x^{\bf r}_{{\bf u}{\bf j}}
    +
    \sum_{\bullet\ne j}(\text{coefficient})\cdot x^{r}_{{\bf u}\bullet}.
  \end{equation*}
  ranging over the $n_1$ ${\bf u}$ will thus produce $n_1$ linearly-independent elements:
  \begin{equation*}
    \dim \left(x\triangleleft C^*\right)\ge n_1
    \quad
    \text{for the dual action \cite[\S 2.2]{dnr}}
    \quad
    x\triangleleft f:=(f\otimes \id)(\Delta(x))
    ,\quad
    \forall f\in C^*.
  \end{equation*}
  The matrix algebra $C^*$ thus has a cyclic module of dimension $\ge n_1$, so must itself have dimension $\ge n_1^2$.
\end{proof}

\begin{corollary}\label{cor:kstepsback}
  Under the hypotheses of \Cref{th:basis4triangchn}, a simple $H({\bf C})$-comodule $V$ whose $k^{th}$ dual $V^{k*}$ ($k\in \bZ_{\ge 0}$) has a simple subquotient whose $\braket{\bZ_{\ge d}}$-valued rank contains the letter $r+k$ has dimension
  \begin{equation*}
    \ge \min\left\{\sqrt{\dim C}\ |\ \text{matrix subcoalgebra }C\le C^r\right\}.
  \end{equation*}
\end{corollary}
\begin{proof}
  Immediate from \Cref{th:largedimatinfty}, given that the $\braket{\bZ_{\ge d}}$-valued rank of such a comodule must contain the letter $r$. 
\end{proof}

\begin{corollary}\label{cor:notsurjfrombij}
  Let ${\bf C}=(C^r,\theta^r)$ be a triangular skew coalgebra chain as in \Cref{th:basis4triangchn}. If
  \begin{equation*}
    \lim_{r\searrow-\infty}\min\dim\left(\text{simple subcoalgebra of }C^r\right)=\infty
  \end{equation*}
  (so that in particular the chain is by necessity bi-infinite) then the only Hopf morphisms $H\to H({\bf C})$ for $H$ with bijective antipode are trivial. 
\end{corollary}
\begin{proof}
  Let $H\xrightarrow{\varphi} H({\bf C})$ be a {\it non}-trivial morphism for bijective-antipode $H$. Some simple $H$-comodule $W$, regarded as an $H({\bf C})$-comodule via $\varphi$, has a simple subquotient $V$ of non-trivial rank ${\bf r}\in \braket{\bZ}$. The antipode of $H$ being bijective, $W$ has iterated right duals
  \begin{equation*}
    W^{-k*}
    ,\quad
    \dim W^{-k*} = \dim W,\quad \forall k\in \bZ_{\ge 0}.
  \end{equation*}
  Now, if $r_0\in \bZ$ is a letter featuring in ${\bf r}$, $r_0-k$ must feature in the rank of some simple subquotient of $W^{-k*}$. As $k\to\infty$ the dimensions of such simples increase indefinitely by \Cref{cor:kstepsback}, contradicting the fact that they can arise as such subquotients.
\end{proof}

\begin{example}\label{ex:notquotbij}
  Instances of \Cref{cor:notsurjfrombij} are easily produced. Take, say
  \begin{itemize}[wide]
  \item $C^r=M_2^{\oplus \bZ}$ and $\theta^r=\id$ for all $r\ge -1$;

  \item whereas for $r<0$ $C^r$ is a sum of countably infinitely many copies of $M_{2^{-r}}$;

  \item and each $\theta^r$, $r<-1$ surjects one $2^{-r}\times 2^{-r}$ onto a sum of two diagonal blocks half the size:
    \begin{equation*}
      \begin{pmatrix}
        *&*\\
        *&*
      \end{pmatrix}
      \xrightarrowdbl{\quad}
      \begin{pmatrix}
        *&0\\
        0&*
      \end{pmatrix}.
    \end{equation*}
  \end{itemize}
  For negative $r$ the simple $C^r$-comodules are $2^{-r}$-dimensional, so that we do fit the framework of \Cref{cor:notsurjfrombij}. 
\end{example}

\begin{remark}\label{re:bratt}
  \Cref{ex:notquotbij} is one particular straightforward way to meet the requirements of \Cref{cor:notsurjfrombij}; it should be clear that the construction admits of much variation. Even restricting attention to cosemisimple coalgebras (as in \Cref{ex:notquotbij}), one can parametrize the surjections $C^r\xrightarrowdbl{} C^{r+1}$ by the {\it Bratteli diagrams} \cite[Chapter 2]{effr_dim} pervasive in the study of {\it AF} ({\it approximately finite} \cite[Chapter 2, p.11]{effr_dim}) operator algebras. 
\end{remark}

%%%%%%%%%%%%%%%%%%%%%%%%%%%%%%%%%%%%%%%%%%%%%%%%%%%%%%%%%%%%%%%%%%%%%%%%%%%%%
%%%%%%%%%%%%%%%%%%%%%%%%%%%%%%%%%%%%%%%%%%%%%%%%%%%%%%%%%%%%%%%%%%%%%%%%%%%%%

\addcontentsline{toc}{section}{References}
%\bibliography{bib}{}
%\bibliographystyle{plain}

\Addresses

\end{document}